\documentclass[reqno,a4paper,12pt]{amsart}

\usepackage{mystyle}
\usepackage{etoolbox}

\usepackage[utf8]{inputenc}
\usepackage[T1]{fontenc}
\usepackage{lmodern}
\usepackage[british]{babel}
\usepackage{microtype}
\usepackage[headings]{fullpage}
\usepackage{parskip}

\usepackage[pdfencoding=auto,pdfusetitle]{hyperref}
\hypersetup{colorlinks=true, linkcolor=blue!30!black, citecolor=green!30!black, urlcolor=red!45!black}

\usepackage{microtype}

\usepackage{amsfonts,amssymb,bbm,stmaryrd}
\usepackage[scr=boondoxo,scrscaled=1.05]{mathalfa}
\usepackage{amsmath}

\usepackage{mathtools}
\usepackage[noabbrev,capitalise]{cleveref}
\usepackage{centernot,float,subcaption,url,array,booktabs,colortbl}
\usepackage[shortlabels]{enumitem}

\usepackage{tikz}
\usetikzlibrary{matrix,positioning}
\usetikzlibrary{calc}
\usetikzlibrary{arrows}
\tikzset{> =stealth}

\newcommand{\addQEDstyle}[2]{\AtBeginEnvironment{#1}{\pushQED{\qed}\renewcommand{\qedsymbol}{#2}}\AtEndEnvironment{#1}{\popQED}}

\theoremstyle{plain}
\newtheorem{theorem}{Theorem}[section]
\newtheorem{lemma}[theorem]{Lemma}
\newtheorem{proposition}[theorem]{Proposition}

\theoremstyle{definition}
\newtheorem{definition}[theorem]{Definition}
\theoremstyle{remark}
\newtheorem{remark}[theorem]{Remark}

\newtheorem{example}[theorem]{Example}
\addQEDstyle{example}{$\triangle$}

\renewcommand{\epsilon}{\varepsilon}
\renewcommand{\phi}{\varphi}

\newcommand{\id}{\mathrm{id}}

\newcommand{\splitext}[6]{%
\tikz[baseline]{
\newdimen{\mylabelwidth}
\newdimen{\skipwidth}
\node[anchor=base] (A) {\hspace*{\dimexpr0.5pt-\pgfkeysvalueof{/pgf/inner xsep}}${#1}$};
\settowidth{\mylabelwidth}{\pgfinterruptpicture {$#2$} \endpgfinterruptpicture}
\pgfmathsetlength{\skipwidth}{max(\mylabelwidth,10pt)}
\node[right] (B) at ([xshift=\skipwidth+12pt]A.east) {${#3}$};
\settowidth{\mylabelwidth}{\pgfinterruptpicture {$#4$} \endpgfinterruptpicture}
\settowidth{\skipwidth}{\pgfinterruptpicture {$#5$} \endpgfinterruptpicture}
\pgfmathsetlength{\skipwidth}{max(\skipwidth,\mylabelwidth,10pt)}
\node[right] (C) at ([xshift=\skipwidth+12pt]B.east) {${#6}$\hspace*{\dimexpr0.5pt-\pgfkeysvalueof{/pgf/inner xsep}}};
\draw[right hook->] (A) to node [above] {${#2}$} (B);
\draw[transform canvas={yshift=0.5ex},->>] (B) to node [above] {${#4}$} (C);
\draw[transform canvas={yshift=-0.5ex},->] (C) to node [below] {${#5}$} (B);
}}

\title{Weakly Schreier extensions for general algebras}

\author[G. Manuell]{Graham Manuell}
\address{CMUC, Department of Mathematics, University of Coimbra, Coimbra, Portugal}
\email{graham@manuell.me}
\thanks{The first author acknowledges financial support from the Centre for Mathematics of the University of Coimbra (UIDB/00324/2020, funded by the Portuguese Government through FCT/MCTES)}

\author[N. Martins-Ferreira]{Nelson Martins-Ferreira}
\address{Instituto Politécnico de Leiria, Leiria, Portugal}
\email{martins.ferreira@ipleiria.pt}
\thanks{The second author's work is supported by Fundação para a Ciência e a Tecnologia FCT/MCTES (PIDDAC) through the following Projects:  Associate Laboratory ARISE LA/P/\-0112/2020; UIDP/04044/2020; UIDB/04\-044/2020; PAMI - ROTEIRO/0328/2013 (Nº 022158); MATIS (CEN\-TRO-01-0145-FEDER-000014 - 3362); Generative.Thermodynamic; by CDRSP and ESTG from the Polytechnic Institute of Leiria.}


\date{5 June 2023}

\begin{document}

\maketitle
\thispagestyle{empty}

\begin{abstract}
Weakly Schreier split extensions are a reasonably large, yet well-understood class of monoid extensions, which generalise some aspects of split extensions of groups.
This short note provides a way to define and study similar classes of split extensions in general algebraic structures (parameterised by a term $\theta$). These generalise weakly Schreier extensions of monoids, as well as general extensions of semi-abelian varieties (using the $\theta$ appearing in their syntactic characterisation). Restricting again to the case of monoids, a different choice of $\theta$ leads to a new class of monoid extensions, more general than the weakly Schreier split extensions.
\end{abstract}

\section{Introduction}

In the category of groups, the group $A$ in a split extension $\splitext{X}{k}{A}{p}{s}{B}$ can be obtained from $X$, $B$ and an action of $B$ on $X$ by the semidirect product construction. The theory of split extensions of groups is important for understanding the structure of groups, building new groups from old ones, and doing homological algebra.
An analogous theory can be developed for any semi-abelian variety \cite{bourn1998protomodularity,bourn2003characterization,clementino2015semidirect}.

Monoids fail to form a semi-abelian variety, since not all split extensions are well-behaved.
However, restricting to the class of \emph{Schreier} split extensions results in a theory which keeps the desirable properties of extensions of groups (see \cite{patchkoria1998crossed,bourn2014schreier,bourn2016monoids}).
If we relax the conditions somewhat, yet more general classes of monoid extensions can be considered including that of \emph{weakly Schreier split extensions} \cite{bourn2015partialMaltsev,faul2022survey}.
\begin{definition}\label{def:Schreier_extension}
 A \emph{weakly Schreier split extension} of monoids is a diagram \[\splitext{X}{k}{A}{p}{s}{B}\] where $k$ is the kernel of $p$, $s$ is a section of $p$ and such that for every $a \in A$ there is an $x \in X$ such that $k(x) \cdot sp(a) = a$. The split extension is \emph{Schreier} if each $x$ is unique.
\end{definition}
These extensions were characterised in \cite{faul2021characterization} and were later studied in \cite{martins2020semibiproduct,martins2023semibiproduct} as `semibiproducts', where a certain retraction map $q$ of the kernel map $k$ is considered as part of the structure.

There has been some interest in extending the notion of Schreier extensions to other varieties of algebra (see \cite{martins2013semidirect,gran2019split}), though this has not yet been carried out for weakly Schreier extensions.
The aim of this paper is to provide a very general definition of weakly Schreier split extension for pointed varieties that recovers the usual theory for monoids, as well as for general semi-abelian varieties.
Our notion is parameterised by a term $\theta$ and by taking a different choice of $\theta$ we can also obtain a more general class of monoid extensions with similar properties.

In a Schreier split extension of monoids $\splitext{X}{k}{A}{p}{s}{B}$, the middle monoid $A$ always has the cartesian product $X \times B$ as its underlying set (but possibly with a different multiplication). Weakly Schreier extensions generalise these by allowing the underlying set of $A$ to be a certain subset (or equivalently a certain quotient) of $X \times B$. We will show in \cref{eg: A=5 XxB=4} that the new class of monoid extensions we define is general enough to include cases where the cardinality of $A$ is strictly \emph{larger} than $X \times B$.

In the final section we show that equipping a category of algebras with the class of $\theta$-weakly-Schreier split extensions always gives an $S$-protomodular category in the weak sense of \cite{bourn2015partialMaltsev}.

\section{Background and motivation}

Group extensions are among the most important concepts in group theory: they allow us to decompose complex groups into simpler parts and to use knowledge about these parts to tell us things about the larger group. Indeed, the ability to decompose groups in this way is the primary motivation for the famous classification of finite simple groups (see \cite{gorenstein2013finite}).

An important class of extensions is given by \emph{split extensions}.
\begin{definition}
 A \emph{split extension} of groups is a diagram \[\splitext{X}{k}{A}{p}{s}{B}\]
 where $p$ is a quotient map, $k$ is the kernel of $p$ and $s$ is a section of $p$.
\end{definition}

Given the kernel group $X$ and the quotient group $B$ in a split extension, the group $A$ is always given by a \emph{semidirect product} $X \rtimes B$, which has underlying set $X \times B$ and multiplication given by $(x_1,b_1) \cdot (x_2,b_2) = (x_1 \alpha(b_1,x_2), b_1 b_2)$ where $\alpha\colon B \times X \to X$ is a group action associated to the extension.
Such a classification is useful for building up groups from smaller ones.

There are also a number of `homological' lemmas that can be proved for group extensions. The following result can be useful for establishing properties of group homomorphisms between groups that can be decomposed as semidirect products.
\begin{lemma}[Split Short Five Lemma]
 Consider the following commutative diagram of groups where the top and bottom rows are split extensions.
 \begin{center}
  \begin{tikzpicture}[node distance=2.0cm, auto]
    \node (A) {$A_1$};
    \node (B) [below of=A] {$A_2$};
    \node (C) [right of=A] {$B_1$};
    \node (D) [right of=B] {$B_2$};
    \node (E) [left of=B] {$X_2$};
    \node (F) [left of=A] {$X_1$};
    \draw[->] (A) to node [swap] {$g$} (B);
    \draw[transform canvas={yshift=0.5ex},->>] (A) to node {$p_1$} (C);
    \draw[transform canvas={yshift=-0.5ex},->] (C) to node {$s_1$} (A);
    \draw[transform canvas={yshift=0.5ex},->>] (B) to node {$p_2$} (D);
    \draw[transform canvas={yshift=-0.5ex},->] (D) to node {$s_2$} (B);
    \draw[->] (C) to node {$h$} (D);
    \draw[right hook->] (F) to node {$k_1$} (A);
    \draw[right hook->] (E) to node [swap] {$k_2$} (B);
    \draw[->] (F) to node [swap] {$f$} (E);
  \end{tikzpicture}
 \end{center}
 If $f$ and $h$ are both (a) injective, (b) surjective, or (c) isomorphisms, then so is $g$.
\end{lemma}

Since group extensions are so useful, there is a strong desire to generalise the concept and the associated results to other algebraic structures. A broad setting in which extensions can be studied is that of \emph{pointed protomodular categories} (see \cite{bourn1998protomodularity}). However, we will be concerned, not with general such categories, but only with varieties of algebras whose categories satisfy this condition.

These were characterised in \cite{bourn2003characterization} as those varieties with a unique constant $0$, an $(n+1)$-ary term $\theta$, and $n$ binary operations $\alpha_1,\dots,\alpha_n$ such that the equations $\alpha_i(x,x) = 0$ and $\theta(\alpha_1(x,y),\dots,\alpha_n(x,y),y) = x$ hold. We call these \emph{semi-abelian varieties}, but they had previously been studied under the name \emph{classically ideal determined varieties} in \cite{ursini1994subtractive}.

One important class of algebras that we would like to be able to handle is that of monoids. However, monoids do not form a semi-abelian variety. One problem is that general (split) extensions of monoids are not well-behaved.
For example, results such as the Split Short Five Lemma do not generally hold for monoids. Nonetheless, there are classes of monoid extension that still behave very similarly to group extensions. The notion of $S$-protomodular category was introduced to describe this situation (see \cite{bourn2016monoids}).

The `$S$' in $S$-protomodular refers to a class of split epimorphisms (equipped with their sections) which are used in the `good' split extensions. The motivating example is the class corresponding to \emph{Schreier split extensions} of monoids (see \cref{def:Schreier_extension}).

Like split extensions of groups, Schreier split extensions can be characterised in terms of a notion of semidirect product and satisfy the Split Short Five lemma.
However, there are many important instances of monoid extensions that are not Schreier (for example, see \cite{faul2021lambda}).
Therefore, the larger class of \emph{weakly Schreier split extensions} has also been studied recently (see \cite{faul2021characterization}).

There is a natural desire to unify the theory of extensions of semi-abelian algebras with that of (weakly) Schreier extensions of monoids, and to be able to handle a wider range of algebraic structures. We will show how to do this, and in the process we will also obtain an even larger class of monoid extensions for which a useful kind of `semidirect product' can be described.

\section{\texorpdfstring{$\theta$}{θ}-weakly-Schreier extensions}

Let $\mathcal{V}$ be a variety of algebras with a unique constant $0$ and a given $(n+1)$-ary term $\theta$ such that $\theta(0,\dots,0,x) = x$.
The category $\mathbb{C}_\mathcal{V}$ of such algebras is a pointed category with zero morphisms given by the constant $0$ maps.
In such a category, kernels, and hence split extensions, can be defined.

\begin{definition}
 We say a split extension \splitext{X}{k}{A}{p}{s}{B} in $\mathbb{C}_\mathcal{V}$ is \emph{$\theta$-weakly-Schreier} if there exist $n$ functions $q_1,\dots,q_n\colon A \to X$ such that $\theta(kq_1(a),\dots,kq_n(a),sp(a)) = a$.
 We may always assume $q_i(0) = 0$.
\end{definition}
Note that these $q_i$ functions are not required to be homomorphisms.
In the case of monoids with $\theta(x,y) = x + y$, this reduces to the usual definition of a weakly Schreier extension (assuming the axiom of choice).

Recall that a semi-abelian variety has a unique constant $0$, an $(n+1)$-ary term $\theta$, and $n$ binary operations $\alpha_1,\dots,\alpha_n$ such that
\begin{itemize}
 \item $\alpha_i(x,x) = 0$,
 \item $\theta(\alpha_1(x,y),\dots,\alpha_n(x,y),y) = x$.
\end{itemize}
This $\theta$ satisfies our above condition, since $\theta(0,\dots,0,x) = \theta(\alpha(x,x),\dots,\alpha(x,x),x) = x$.
In this case \emph{every} split extension \splitext{X}{k}{A}{p}{s}{B} is $\theta$-weakly-Schreier by taking \[k q_i(a) = \alpha_i(a,sp(a)).\] Note that $\alpha_i(a,sp(a))$ indeed lies in the kernel of $p$, since we have $p(\alpha_i(a,sp(a))) = \alpha_i(p(a),psp(a)) = \alpha_i(p(a),p(a)) = 0$.

The paper \cite{clementino2015semidirect} describes the semidirect product for semi-abelian varieties. We can use a similar approach to construct the middle algebra $A$ of a $\theta$-weakly-Schreier split extension in our more general setting. 

Consider a $\theta$-weakly-Schreier split extension \splitext{X}{k}{A}{p}{s}{B}.
We define a function $\psi\colon A \to X^n \times B$ by \[\psi(a) = (q_1 a,\dots,q_n a, p a)\]
and a function $\phi\colon X^n \times B \to A$ by \[\phi(x_1,\dots,x_n,b) = \theta(k x_1,\dots,k x_n, sb).\]
Note that here and elsewhere we omit the brackets around the arguments of the functions to avoid the proliferation of brackets.

These functions induce maps between the above extension and the `trivial' split extension of pointed sets \splitext{X^n}{\iota_{X^n}}{X^n \times B}{\pi_B}{\iota_B}{B} where $\iota_{X^n}(\vec{x}) = (\vec{x},0)$, $\pi_B(\vec{x},b) = b$ and $\iota_B(b) = (0,\dots,0,b)$. (Here we write $\vec{x}$ for the tuple of elements $(x_1,\dots,x_n)$.)
In the following diagram $\phi_X$ and $\psi_X$ are obtained from $\phi$ and $\psi$ from the universal property of the kernel.
It is easy to see that the other relevant squares commute by the definition of $\psi$ and $\phi$ and the assumption on $\theta$.
\begin{center}
   \begin{tikzpicture}[node distance=2cm, auto]
    \node (A) {$X$};
    \node (B) [right of=A] {$A$};
    \node (C) [right of=B] {$B$};
    \node (D) [below of=A] {$X^n$};
    \node (E) [right of=D] {$X^n \times B$};
    \node (F) [right of=E] {$B$};
    \draw[->] (A) to node {$k$} (B);
    \draw[transform canvas={yshift=0.5ex},->] (B) to node {$p$} (C);
    \draw[transform canvas={yshift=-0.5ex},->] (C) to node {$s$} (B);
    \draw[->] (D) to [swap] node {$\iota_{X^n}$} (E);
    \draw[transform canvas={yshift=0.5ex},->] (E) to node {$\pi_B$} (F);
    \draw[transform canvas={yshift=-0.5ex},->] (F) to node {$\iota_B$} (E);
    \draw[transform canvas={xshift=-0.5ex},->] (B) to [swap] node {$\psi$} (E);
    \draw[transform canvas={xshift=0.5ex},->] (E) to [swap] node {$\phi$} (B);
    \draw[transform canvas={xshift=-0.5ex},->] (A) to [swap] node {$\psi_X$} (D);
    \draw[transform canvas={xshift=0.5ex},->] (D) to [swap] node {$\phi_X$} (A);
    \draw[double equal sign distance] (C) to (F);
   \end{tikzpicture}
\end{center}

Note that $\phi \psi = \id_A$, since $\phi(\psi(a)) = \theta(kq_1(a),\dots,kq_n(a),sp(a)) = a$ by the $\theta$-weakly-Schreier condition.

On the other hand, we find that the $i^\text{th}$ component of
$\psi \phi (x_1,\dots,x_n,b)$ is equal to $q_i(\theta(k x_1,\dots,k x_n,s b))$ for $i = 1,\dots,n$, while the last component is $ps(b) = b$. Thus, $A$ maps bijectively onto the subset $Y$ of $X^n \times B$ given by \[Y = \{(x_1,\dots,x_n,b) \in X^n \times B \mid q_i(\theta(k x_1,\dots,k x_n,s b)) = x_i \text{ for each $i$} \}.\]

The operations on $A$ can then be transported along this bijection so that an $m$-ary operation $\omega$ can be computed in $Y$ by
\begin{align*}
 [\omega( &(x^1_1,\dots,x^1_n,b^1), \dots, (x^m_1,\dots,x^m_n,b^m) )]_i \\
 &= [\psi \omega( \phi(x^1_1,\dots,x^1_n,b^1), \dots, \phi(x^m_1,\dots,x^m_n,b^m) )]_i \\
 &= q_i(\omega(\theta(k x^1_1,\dots,k x^1_n,s b^1), \dots, \theta(k x^m_1,\dots,k x^m_n,s b^m)))
\end{align*}
and
\begin{displaymath}
 [\omega( (x^1_1,\dots,x^1_n,b^1), \dots, (x^m_1,\dots,x^m_n,b^m) )]_{n+1} = \omega(b^1,\dots,b^m).
\end{displaymath}

If $\omega$ and $\theta$ commute then the first formula simplifies to
\begin{align*}
 [\omega( &(x^1_1,\dots,x^1_n,b^1), \dots, (x^m_1,\dots,x^m_n,b^m) )]_i = \\
 &= q_i(\theta(k\omega(x^1_1, \dots, x^m_1),\dots,k\omega(x^1_n, \dots, x^m_n), s \omega(b^1, \dots, b^m))) \\
 &= \omega(x^1_i, \dots, x^m_i).
\end{align*}
In particular, the $0$ in $Y$ is simply $(0,\dots,0,0)$.
However, in general there is not anything more that we can do to simplify this further.

The corresponding extension is \splitext{X}{k'}{Y}{\pi_B}{\iota_B}{B} where the kernel map $k'$ is given by $k'(x) = \psi k(x) = (q_1k(x),\dots,q_nk(x),0)$.
This map can also be defined without reference to the $q_i$ maps: for $(\vec{y},0) \in Y$ note that $(\vec{y},0) = k'(x)$ if and only if $\phi(\vec{y},0) = \phi k'(x)$. But $\phi k'(x) = k(x)$ and $\phi(\vec{y},0) = \theta(k(y_1),\dots,k(y_n),0) = k \theta(\vec{y},0)$, so that we have $(\vec{y},0) = k'(x)$ if and only if $\theta(\vec{y},0) = x$. Hence, $k'(x)$ is the unique element $(\vec{y},0) \in Y$ such that $\theta_X(\vec{y},0) = x$.
Also note that, by the definition of $\psi$, transporting the $q_i$ maps along the isomorphism, simply gives the restrictions of the product projections $\pi_i\colon X^n \times B \to X$.

In summary, we have the following proposition.
\begin{proposition}
A $\theta$-weakly-Schreier extension \splitext{X}{k}{A}{p}{s}{B} is isomorphic to an extension \splitext{X}{k'}{Y}{\pi_B}{\iota_B}{B}, where
$Y$ is certain subset of $X^n \times B$ and $k'(x)$ is the unique element $(\vec{y},0) \in Y$ such that $\theta_X(\vec{y},0) = x$.
 
Suppose the maps $q_1,\dots,q_n$ witness the $\theta$-weakly-Schreier condition. For each basic operation $\omega$ in the variety we define a map $\vec{\gamma}_\omega\colon (X^n \times B)^m \to X^n$ by
\begin{align*}
 (\vec{\gamma}_\omega(&x^1_1,\dots,x^1_n,b^1, \dots, x^m_1,\dots,x^m_n,b^m))_i \\ &= q_i(\omega(\theta(k x^1_1,\dots,k x^1_n,s b^1), \dots, \theta(k x^m_1,\dots,k x^m_n,s b^m))).
\end{align*}
The maps $\vec{\gamma}_\omega$ for composite terms $\omega$ can be defined in terms of these as necessary.

Then $Y$ is given by \[\{(\vec{x},b) \in X^n \times B \mid \vec{\gamma}_{\id}(\vec{x},b) = \vec{x} \},\]
or, equivalently, by
\[\{(\vec{x},b) \in X^n \times B \mid \vec{\gamma}_{\theta}(\vec{0},\vec{x},b) = \vec{x}\},\]
equipped with operations $\omega_Y$ defined by \[\omega_Y(\vec{x}^1,b^1,\dots,\vec{x}^m,b^m) = (\vec{\gamma}_\omega(\vec{x}^1,b^1,\dots,\vec{x}^m,b^m),\omega_B(b^1,\dots,b^m)).\]
\end{proposition}
\begin{remark}
In fact, as can be seen by simply expanding the definitions, we can define the subset $Y \subseteq X^n \times B$ using $\vec{\gamma}_\omega$ for any term $\omega$ satisfying $\omega(\vec{0},x) = x$ in the variety.
This is sometimes more convenient if $\theta$ is built out of many operations.
\end{remark}

We might now ask what conditions we need to impose on arbitrary $\vec{\gamma}_\omega$ maps for this construction to give a valid $\theta$-weakly-Schreier extension.
\begin{itemize}
 \item Firstly, we need to impose conditions on the $\vec{\gamma}_\omega$ such that defining axioms of the algebra hold (when restricted to $Y \subseteq X^n \times B$).
 \item Secondly, we must ensure well-definedness of $k'$ by requiring that for all $x \in X$ there is a unique $\vec{y} \in X^n$ such that
$(\vec{y},0) \in Y$ and $\theta_X(\vec{y},0) = x$, or equivalently such that $\vec{\gamma}_\theta(\vec{0},\vec{y},0) = \vec{y}$ and $\theta_X(\vec{y},0) = x$.
 \item Next we guarantee that $k'$ is a homomorphism. Note that since $k'$ is bijective onto the restriction of its codomain to $\{(\vec{y},b) \in Y \mid b = 0\}$, it suffices to ensure the inverse $(\vec{y},0) \mapsto \theta_X(\vec{y},0)$ is a homomorphism.
 This holds when $\theta_X(\vec{\gamma}_\omega(\vec{x}^1,0,\dots,\vec{x}^m,0),0) = \omega_X(\theta_X(\vec{x}^1,0),\dots,\theta_X(\vec{x}^m,0))$ for all $\vec{x}^1,\dots,\vec{x}^m$ such that $\vec{\gamma}_\theta(\vec{0},\vec{x}^i,0) = \vec{x}^i$.
 It then follows from the definition of $k'$ that it is a kernel: clearly $\pi_B k' = 0$, while if $(\vec{y},0) \in Y$ then setting $x = \theta_X(\vec{y},0)$ we have $k'(x) = \vec{y}$.
 \item Finally, we add a condition that forces the extension to be $\theta$-weakly-Schreier (with the product projections as the $q_i$ maps). This means $\vec{\gamma_\theta}(k'(x_1),\dots,k'(x_n),\iota_B(b)) = (x_1,\dots,x_n)$ for $(x_1,\dots,x_n,b) \in Y$. It is more convenient to write this in terms of the elements $\vec{y}^i$ such that $(\vec{y}^i,0) = k'(x_i)$. Then each $x_i = \theta_X(\vec{y}^i,0)$ and so the condition becomes that $\vec{\gamma_\theta}(\vec{y}^1,0,\dots,\vec{y}^n,0,\vec{0},b) = (\theta_X(\vec{y}^1,0),\dots,\theta_X(\vec{y}^n,0))$ for all $\vec{y}^1,\dots,\vec{y}^n \in X^n$ such that $\vec{\gamma}_\theta(\vec{0},\vec{y}^i,0) = \vec{y}^i$ and $\vec{\gamma}_\theta(\vec{0},\theta_X(\vec{y}^1,0),\dots,\theta_X(\vec{y}^n,0),b) = (\theta_X(\vec{y}^1,0),\dots,\theta_X(\vec{y}^n,0))$.
\end{itemize}

\section{Examples}

We have already noted that the case of semi-abelian varieties is captured by this theory.

Another core example is, of course, that of weakly Schreier extensions of monoids. These are simply $\theta$-weakly-Schreier extensions with $\theta$ being the monoid operation $\theta(x,y) = x + y$. Our characterisation agrees with that in \cite{faul2021characterization}, though our approach is more similar to that taken in \cite{martins2020semibiproduct}.

\begin{remark}\label{rem:schreier}
Note that we might ask if \emph{Schreier} extensions of monoids have a generalisation to general algebras. This would mean requiring that the $q_i$ maps be \emph{unique}.
This is equivalent to asking the map $\phi\colon X^n \times B \to A$ to be injective (and hence bijective).
However, do note that for $n \ge 2$ this seldom holds even for semi-abelian varieties.
For example, observe that Heyting semilattices form a semi-abelian variety for $n = 2$ with $\theta(x,y,z) = (x \Rightarrow z) \wedge y$ (see \cite{johnstone2004heyting}) and consider the extension of Heyting semilattices given by $\{a \le 1 \} \hookrightarrow \{ 0 \le a \le 1 \} \twoheadrightarrow \{ [0] \le [a] = [1] \}$ with section defined by $[0] \mapsto 0$ and $[1] \mapsto 1$.
If $q_1(x) = x \Rightarrow sp(a)$, then \emph{any} map $q_2$ such that $q_2(1) = 1$ and $q_2(a) = a$ witnesses the weakly Schreier condition.
See \cite[Theorem 4.2]{clementino2015semidirect} or \cite[Corollary 3.2]{gray2015algebraic} for a characterisation of the semi-abelian varieties for which every split extension is Schreier in this sense.
\end{remark}

Schreier split extensions of unital magmas were studied in \cite{gran2019split}. Taking $\theta$ to be the magma operation we recover a more general class of `\emph{weakly} Schreier' split extensions of magmas. The Schreier extensions satisfy an additional condition $q_1(\theta(k x,s b)) = x$, as explained in \cite[pg 6]{martins2022magma}, that forces the underlying set of the middle magma to be the cartesian product of the other two (as in \cref{rem:schreier}).

Our main new example is again monoids, but this time with $\theta$ given by the ternary term $\theta(x,y,z) = x + z + y$.
In this case, an extension \splitext{X}{k}{A}{p}{s}{B} is $\theta$-weakly-Schreier if and only if there are $q_1,q_2 \colon A \to X$ such that $a = kq_1(a) + sp(a) + kq_2(a)$.
This is a strictly more general class than the usual class of weakly Schreier split extensions of monoids, which corresponds to requiring $q_2 \equiv 0$. 

In particular, we note that in contrast to what happens with the usual notion of weakly Schreier extensions, the middle monoid $A$ here can have a cardinality strictly larger than that of $X \times B$.
\begin{example}\label{eg: A=5 XxB=4}
Suppose $A$ is given by the following multiplication table.
\begin{center}
\begin{tabular}{>{$}l<{$}|*{5}{>{$}l<{$}}}
\oplus & 0 & 1 & 2 & 3 & 4 \\
\hline\vrule height 12pt width 0pt
0 & 0 & 1 & 2 & 3 & 4 \\
1 & 1 & 1 & 4 & 4 & 4 \\
2 & 2 & 3 & 2 & 3 & 4 \\
3 & 3 & 3 & 4 & 4 & 4 \\
4 & 4 & 4 & 4 & 4 & 4 \\
\end{tabular}
\end{center}
Let $X = B = \{0,1\}$ with saturating addition, let $k$ be the obvious inclusion, let $p$ be defined by

\[p(a) = \begin{cases}
 0 & \text{if $a \le 1$} \\
 1 & \text{if $a \ge 2$},
\end{cases}\]
and suppose $s\colon 1 \mapsto 2$.
This split extension is $\theta$-weakly-Schreier with
\begin{align*}
 q_1(a) &= \begin{cases}
 1 & \text{if $a = 4$} \\
 0 & \text{otherwise}
\end{cases} \\
q_2(a) &= \begin{cases}
 1 & \text{if $a = 1$ or $a = 3$} \\
 0 & \text{otherwise}.
\end{cases}
\end{align*}
Note that $|A| = 5$, while $|X \times B| = 4$.
\end{example}

For this class of monoid extension we can use associativity to say more about the `action' given by $\vec{\gamma}_+$.
Given an extension, this is defined as
\begin{align*}
 (\vec{\gamma}_+(x_1^1,x_2^1,b^1,x_1^2,x_2^2,b^2))_i
  &= q_i(k x_1^1 + s b^1 + k x_2^1 + k x_1^2 + s b^2 + k x_2^2) \\
  &= q_i(k x_1^1 + s b^1 + k(x_2^1 + x_1^2) + s b^2 + k x_2^2).
\end{align*}
Now suppose we set $u = s b^1 + k(x_2^1 + x_1^2) + s b^2$. By the $\theta$-weakly-Schreier condition this can be expressed as $u = kq_1(u) + sp(u) + kq_2(u) = kq_1(u) + s(b^1+b^2) + kq_2(u)$. We then find
\begin{align*}
 (\vec{\gamma}_+(x_1^1,x_2^1,b^1,x_1^2,x_2^2,b^2))_i
  &= q_i(k x_1^1 + u + k x_2^2) \\
  &= q_i(k x_1^1 + kq_1(u) + s(b^1+b^2) + kq_2(u) + k x_2^2) \\
  &= q_i(k(x_1^1 + q_1(u)) + s(b^1+b^2) + k(q_2(u) + x_2^2)).
\end{align*}
In this way $\vec{\gamma}_+$ can be decomposed into four simpler maps.
Set
\begin{align*}
 \sigma_i(b,x,b') &= q_i(s b + k x + s b'), \\
 \tau_i(x,b,x') &= q_i(k x + s b + k x').
\end{align*}
Then
\[
 (\vec{\gamma}_+(x_1^1,x_2^1,b^1,x_1^2,x_2^2,b^2))_i = \tau_i(x_1^1 + \sigma_1(b^1,x_2^1+x_1^2,b^2), b^1+b^2, \sigma_2(b^1,x_2^1+x_1^2,b^2)+x_2^2).
\]
Finally, in this case it is convenient to define $Y$ to be $\{(\vec{x},b) \in X^2 \times B \mid \vec{\gamma}_+(\vec{0},0,\vec{x},b) = \vec{x} \}$.

It is then routine to compute the axioms $\sigma_{1,2}(b,x,b')$ and $\tau_{1,2}(x,b,x')$ to obtain a characterisation of this class of extensions, though we will omit the details.

\section{Properties of these extensions}

Schreier split extensions of monoids are very well-behaved and were the inspiration for the notion of an $S$-protomodular category \cite{bourn2016monoids}. An $S$-protomodular category has a class of split epimorphisms whose corresponding split extensions behave similarly to general split extensions in a protomodular category. In \cite{bourn2015partialMaltsev} Bourn defines a slightly weaker notion of $S$-protomodularity by omitting a certain completeness condition. This weaker definition captures \emph{weakly} Schreier extensions of monoids. Our $\theta$-weakly-Schreier extensions also give rise to weakly $S$-protomodular categories in this sense.

\begin{proposition}
 Let $\mathcal{V}$ be a pointed variety with term $\theta$ as above. The category $\mathbb{C}_\mathcal{V}$ of these algebras together equipped with the class $\Sigma$ of split epimorphisms coming from $\theta$-weakly-Schreier extensions is a weakly $\Sigma$-protomodular category.
\end{proposition}
\begin{proof}
 Certainly the category $\mathbb{C}_\mathcal{V}$ is finitely complete and pointed. We require that for every $(p,s) \in \Sigma$, $\ker p$ and $s$ are jointly extremally epic, that $\Sigma$ contains all isomorphisms, and that $\Sigma$ is stable under pullback.
 
 The first condition is immediate from the definition of $\theta$-weakly-Schreier property. For the second, note that if $p$ and $s$ are mutual inverses, then $\theta(k(0),\dots,k(0),sp(a)) = \theta(0,\dots,0,a) = a$ by the assumption on $\theta$ and hence taking $q_1 = \dots = q_n \equiv 0$ we have $(p,s) \in \Sigma$.
 
 For pullback stability, consider the diagram
 \begin{center}
  \begin{tikzpicture}[node distance=2.5cm, auto]
    \node (A) {$A \times_{\scriptscriptstyle B} B'$};
    \node (B) [below of=A] {$A$};
    \node (C) [right of=A] {$B'$};
    \node (D) [right of=B] {$B$};
    \node (E) [left of=B] {$X$};
    \node (F) [left of=A] {$X$};
    \draw[->] (A) to node [swap] {$f'$} (B);
    \draw[transform canvas={yshift=0.5ex},->>] (A) to node {$p'$} (C);
    \draw[transform canvas={yshift=-0.5ex},->] (C) to node {$s'$} (A);
    \draw[transform canvas={yshift=0.5ex},->>] (B) to node {$p$} (D);
    \draw[transform canvas={yshift=-0.5ex},->] (D) to node {$s$} (B);
    \draw[->] (C) to node {$f$} (D);
    \draw[right hook->] (F) to node {$k'$} (A);
    \draw[right hook->] (E) to node [swap] {$k$} (B);
    \draw[double equal sign distance] (F) to (E);
    \begin{scope}[shift=({A})]
        \draw +(0.3,-0.7) -- +(0.7,-0.7) -- +(0.7,-0.3);
    \end{scope}
  \end{tikzpicture}
 \end{center}
 where $(A \times_{\scriptscriptstyle B} B', f', p')$ is the pullback of $f$ and $p$
 and the bottom row is a $\theta$-weakly-Schreier extension.
 Note that by the universal property of the pullback $s'$ is the unique section of $p'$ such that $f's' = sf$.
 We see that $X$ is the kernel of $p'\colon (a,b') \mapsto b'$ since $(a,0) \in A \times_{\scriptscriptstyle B} B'$ if only if $p(a) = f(0) = 0$.
 
 If $q_1,\dots,q_n\colon A \to X$ are the associated maps for the bottom extension, we define maps $q'_1,\dots,q'_n \colon A \times_{\scriptscriptstyle B} B' \to X$ by $q'_i\colon (a,b') \mapsto q_i(a)$.
 Then
 \begin{align*}
  \theta_{A \times_{\scriptscriptstyle B} B'}(& k'q'_1(a,b'),\dots,k'q'_n(a,b'),s'p'(a,b')) \\
  &=
  \theta_{A \times_{\scriptscriptstyle B} B'}(k'q_1 a,\dots,k'q_n a,s'b') \\
  &= (\theta_A (k q_1 a,\dots,k q_n a,sf b'), \theta_{B'}(0,\dots,0,b')) \\
  &= (\theta_A (k q_1 a,\dots,k q_n a,sp a), b') \\
  &= (a,b'),
 \end{align*}
 as required.
\end{proof}

Note that this weak form of $\Sigma$-protomodularity does not imply an analogue of the Split Short Five Lemma. However, we do have the following result for surjections generalising the case of weakly Schreier extensions of monoids.
\begin{lemma}
 Let $\mathbb{C}$ be a pointed weakly $\Sigma$-protomodular category and consider a morphism of extensions from the distinguished class $\Sigma$ as in the following diagram.
 \begin{center}
  \begin{tikzpicture}[node distance=2.0cm, auto]
    \node (A) {$A_1$};
    \node (B) [below of=A] {$A_2$};
    \node (C) [right of=A] {$B_1$};
    \node (D) [right of=B] {$B_2$};
    \node (E) [left of=B] {$X_2$};
    \node (F) [left of=A] {$X_1$};
    \draw[->] (A) to node [swap] {$g$} (B);
    \draw[transform canvas={yshift=0.5ex},->>] (A) to node {$p_1$} (C);
    \draw[transform canvas={yshift=-0.5ex},->] (C) to node {$s_1$} (A);
    \draw[transform canvas={yshift=0.5ex},->>] (B) to node {$p_2$} (D);
    \draw[transform canvas={yshift=-0.5ex},->] (D) to node {$s_2$} (B);
    \draw[->] (C) to node {$h$} (D);
    \draw[right hook->] (F) to node {$k_1$} (A);
    \draw[right hook->] (E) to node [swap] {$k_2$} (B);
    \draw[->] (F) to node [swap] {$f$} (E);
  \end{tikzpicture}
 \end{center}
 If $f$ and $h$ are extremal epimorphisms, then so is $g$.
\end{lemma}
\begin{proof}
 By the condition on $\Sigma$, the maps $k_2$ and $s_2$ are jointly extremally epic.
 Since $\mathbb{C}$ has pullbacks, extremal epimorphisms are closed under composition. In fact, it can be shown that composites of jointly extremally epic maps with extremal epimorphisms are still jointly extremally epic.
 Thus, $k_2 f = g k_1$ and $s_2 h = g s_1$ are jointly extremally epic.
 It follows that $g$ is extremally epic, as required.
\end{proof}

Finally, in good cases we would probably expect products $X \times B$ to give rise to $\theta$-weakly-Schreier split extensions. By pullback stability, this holds if and only if the extension $X \xrightarrow{\id} X \to 1$ is a $\theta$-weakly-Schreier extension. This holds whenever there are maps $q_1,\dots,q_n\colon X \to X$ such that $\theta(q_1(x),\dots,q_n(x),0) = x$ for all $x \in X$.
We note that this condition is satisfied in all our examples,
but fails for the variety of left-unital magmas (taking $\theta$ to be the product).

\bibliographystyle{abbrv}
\bibliography{references}

\begin{thebibliography}{10}

\bibitem{bourn2015partialMaltsev}
D.~Bourn.
\newblock Partial {M}al'tsevness and partial protomodularity.
\newblock arXiv preprint
  \href{https://arxiv.org/abs/1507.02886v1}{arXiv:1507.02886v1}, 2015.

\bibitem{bourn1998protomodularity}
D.~Bourn and G.~Janelidze.
\newblock Protomodularity, descent, and semidirect products.
\newblock {\em Theory Appl. Categ.}, 4(2):37--46, 1998.

\bibitem{bourn2003characterization}
D.~Bourn and G.~Janelidze.
\newblock Characterization of protomodular varieties of universal algebras.
\newblock {\em Theory Appl. Categ.}, 11(6):143--447, 2003.

\bibitem{bourn2014schreier}
D.~Bourn, N.~Martins-Ferreira, A.~Montoli, and M.~Sobral.
\newblock {S}chreier split epimorphisms between monoids.
\newblock {\em Semigroup Forum}, 88(3):739--752, 2014.

\bibitem{bourn2016monoids}
D.~Bourn, N.~Martins-Ferreira, A.~Montoli, and M.~Sobral.
\newblock Monoids and pointed {$S$}-protomodular categories.
\newblock {\em Homology, Homotopy Appl.}, 18(1):151--172, 2016.

\bibitem{clementino2015semidirect}
M.~M. Clementino, A.~Montoli, and L.~Sousa.
\newblock Semidirect products of (topological) semi-abelian algebras.
\newblock {\em J. Pure Appl. Algebra}, 219(1):183--197, 2015.

\bibitem{faul2021characterization}
P.~F. Faul.
\newblock A characterization of weakly {S}chreier extensions of monoids.
\newblock {\em J. Pure Appl. Algebra}, 225(2):106489, 2021.

\bibitem{faul2021lambda}
P.~F. Faul.
\newblock $\lambda$-semidirect products of inverse monoids are weakly
  {S}chreier extensions.
\newblock {\em Semigroup Forum}, 102(2):422--436, 2021.

\bibitem{faul2022survey}
P.~F. Faul.
\newblock A survey of {S}chreier-type extensions of monoids.
\newblock {\em Semigroup Forum}, 104(3):519--539, 2022.

\bibitem{gorenstein2013finite}
D.~Gorenstein.
\newblock {\em Finite simple groups: an introduction to their classification}.
\newblock Springer, 2013.

\bibitem{gran2019split}
M.~Gran, G.~Janelidze, and M.~Sobral.
\newblock Split extensions and semidirect products of unitary magmas.
\newblock {\em Comment. Math. Univ. Carolin.}, 60(4):509--527, 2019.

\bibitem{gray2015algebraic}
J.~R.~A. Gray and N.~Martins-Ferreira.
\newblock On algebraic and more general categories whose split epimorphisms
  have underlying product projections.
\newblock {\em Appl. Categ. Structures}, 23(3):429--446, 2015.

\bibitem{johnstone2004heyting}
P.~Johnstone.
\newblock A note on the semiabelian variety of {H}eyting semilattices.
\newblock In G.~Janelidze, B.~Pareigis, and W.~Tholen, editors, {\em Galois
  Theory, Hopf Algebras, and Semiabelian Categories}, volume~43 of {\em Fields
  Institute Communications}, pages 317--318. American Mathematical Society,
  2004.

\bibitem{martins2020semibiproduct}
N.~Martins-Ferreira.
\newblock Semi-biproducts in monoids.
\newblock arXiv preprint
  \href{https://arxiv.org/abs/2002.05985}{arXiv:2002.05985}, 2020.

\bibitem{martins2022magma}
N.~Martins-Ferreira.
\newblock On semibiproducts of magmas and semigroups.
\newblock arXiv preprint
  \href{https://arxiv.org/abs/2208.12704}{arXiv:2208.12704}, 2022.

\bibitem{martins2023semibiproduct}
N.~Martins-Ferreira.
\newblock Pointed semibiproducts of monoids.
\newblock {\em Theory Appl. Categ.}, 39(6):172--185, 2023.

\bibitem{martins2013semidirect}
N.~Martins-Ferreira, A.~Montoli, and M.~Sobral.
\newblock Semidirect products and crossed modules in monoids with operations.
\newblock {\em J. Pure Appl. Algebra}, 217(2):334--347, 2013.

\bibitem{patchkoria1998crossed}
A.~Patchkoria.
\newblock Crossed semimodules and {S}chreier internal categories in the
  category of monoids.
\newblock {\em Georgian Math. J.}, 5(6):575--581, 1998.

\bibitem{ursini1994subtractive}
A.~Ursini.
\newblock On subtractive varieties, {I}.
\newblock {\em Algebra Universalis}, 31:204--222, 1994.

\end{thebibliography}

\end{document}